\documentclass[12pt]{amsart}
\usepackage[usenames]{color}
\usepackage{ulem}  %% added Hln
\usepackage{url} %% added Hln
\usepackage{marginnote} %% added Hln
\usepackage{amscd,amsmath,amssymb,amsfonts}
\usepackage[cmtip, all]{xy}
\usepackage{enumerate}
\theoremstyle{plain}
\newtheorem{thm}[equation]{Theorem}

\newtheorem{lem}[equation]{Lemma}
\newtheorem{cor}[equation]{Corollary}

\newtheorem{prop}[equation]{Proposition}
\newtheorem{claim}[equation]{Claim}

\newtheorem{conj}[equation]{Conjecture}

\theoremstyle{definition}

\newtheorem{rmk}[equation]{Remark}
\newtheorem{rmks}[equation]{Remarks}
\numberwithin{equation}{section}

\newcommand{\isoarrow}{{~\overset\sim\longrightarrow~}}%\to\longrightarrow~}}

\newcommand{\CO}{{\mathcal{O}}}

\newcommand{\ZZ}{{\mathbb Z}}

\newcommand{\CE}{{\mathcal{E}}}

\newcommand{\CA}{{\mathcal{A}}}

\newcommand{\fg}{{\mathfrak{g}}}

\newcommand{\ra}{{~\rightarrow~}}

\newcommand{\CH}{{\mathcal{H}}}
\newcommand{\hX}{{\hat{X}}}

\newcommand{\QQ}{{\mathbb Q}}

\newcommand{\Ql}{{\mathbb Q}_{\ell}}

\newcommand{\Qbar}{{\overline{\mathbb Q}}}

\newcommand{\RR}{{\mathbb R}}
\newcommand{\ad}{{\mathbf A}}
\newcommand{\af}{{{\mathbf A}_f}}

\newcommand{\CC}{{\mathbb C}}

\newcommand{\Vect}{{\rm Vect}}  %% added Hln
\newcommand{\Rep}{{\rm Rep}} %% added Hln
\newcommand{\Gal}{{\rm Gal}}
\newcommand{\ml}[2]{\begin{multline}\label{#1}#2 \end{multline}} %% added Hln
\newcommand{\ga}[2]{\begin{gather}\label{#1}#2 \end{gather}} %% added Hln

\title [Chern classes of automorphic vector bundles]
{Chern classes of automorphic vector bundles}

\author[H\'el\`ene Esnault]{H\'el\`ene Esnault}
\address{Freie Universit\"at Berlin, Mathematik, Arnimallee 3, 14195 Berlin, Germany}
\email{esnault@math.fu-berlin.de}
\author[Michael Harris]{Michael Harris}
\address{    Department of Mathematics,
Columbia University,   2990 Broadway
New York, New York 10027, USA }
\email{harris@math.columbia.edu}
\thanks{The research leading to these results has received funding from the European Research Council under the European Community's Seventh Framework Programme (FP7/2007-2013) / ERC Advanced Grants  22224001 and 290766, the Einstein Foundation,    the NSF grant DMS- 1404769}

\begin{document}
\maketitle
\hfill{{\it Dedicated to Yuri Ivanovich Manin on the occasion of his 80-th birthday}}
\begin{abstract}
We prove that Chern classes in continuous $\ell$-adic cohomology of automorphic bundles associated to representations of $G$ on a projective Shimura variety with data $(G,X)$ are trivial rationally.  It is a consequence of Beilinson's conjectures which predict that the Chern classes in the Chow groups vanish rationally.

\end{abstract}
\section*{Introduction} \label{sec:intro}
Let $X$ be a smooth projective variety defined over a number field $k$.  Beilinson  \cite[Conj.~2.4.2.1]{Bei85} conjectures %\mar{I made it tighter}
that the rational Chow ring $CH(X)_{\QQ} $  injects into Deligne cohomology of $X\otimes_k \mathbb{C}$.   Concretely, 
 if a class in $CH^n(X)_{\QQ}$ vanishes in $H^{2n}_{\mathcal{D}}(X, \QQ(n))$, it is expected to be $0$.
 There is not a single  example with dimension $X \ge 2$ with large Chow ring   $CH(X\otimes_k \mathbb{C})$ for which this conjecture has been verified. 

\medskip

On the other hand, there are Chern classes reflecting the fact that $X$ is defined over a number field. On proper models $\mathcal{X}_U$ over a non-trivial open $U$ of ${\rm Spec}(\ZZ)$, one has Chern classes in $\ell$-adic cohomology $H^{2n}(\mathcal{X}_U, \QQ_\ell(n))$.  By taking the inductive limit over such $\mathcal{X}_U$, these yield 
the Chern classes in continuous $\ell$-adic cohomology $H^{2n}_{\rm cont}( X, \QQ(n))$ (\cite[Section~2]{Jan87}). 
Each space $H^{2n}_{\rm cont}( X, \QQ(n))$ is filtered by the abutment of the Hochshild-Serre spectral sequence, which, by Deligne's argument  \cite[Thm.~1.5]{Del68} using the strong Lefschetz theorem
\cite[Thm.~4.4.1]{Del80}, degenerates at $E_2$. Given that  $H^1(k, H^{2n-1}(X_{\bar k}, \QQ_\ell(n)))$,  the first graded piece of the filtration,  can be interpreted as the extension group of $\QQ_\ell(0)$ by  $H^{2n-1}(X_{\bar k}, \QQ_\ell(n))$ in the category of Galois modules -- just as $H^{2n-1}_{\mathcal{D}}( X_{ \mathbb{C}}, \QQ(n))$ can be interpreted as 
 the extension group of $\QQ(0)$ by  $H^{2n-1}(X_{\mathbb{C}}, \QQ(n))$ in the category of  Hodge structures over $\QQ$ -- Beilinson's conjecture predicts that
 \begin{conj} \label{conj:sasha} With notation as above, 
  if a class in $CH^n(X)_{\QQ}$ vanishes in 
 $H^i(k, H^{2n-i}(X_{\bar k}, \QQ_\ell(n))$ for $i=0,1$, then it does for $i=2$.
 \end{conj}  
 This conjecture seems to be more modest than the general motivic  one above.  It is a  fascinating problem in Galois cohomology, and it  hasn't been studied at all. An analogous question for function fields over finite fields has been considered and proved to be true for $0$-cycles in \cite[Thm.~0.1]{Ras95}. 
 
 \medskip

On the other hand, the Chern classes of flat bundles on a smooth projective  variety $X$ defined over $\mathbb{C}$ vanish in Deligne cohomology $H^{2n}_{\mathcal{D}}(X, \mathbb{Q}(n))$ for $n\ge 2$, due to Reznikov's theorem  \cite[Thm.~1.1]{Rez94}, giving a positive answer to Bloch's conjecture  \cite[Intro.]{Blo77}.   In particular,  Beilinson's conjecture implies that  the Chern classes of flat bundles  on a smooth projective  variety $X$ defined over a number field 
vanish in the Chow groups $CH^n(X)_{\QQ}$ for $n\ge 2$.

\medskip

In addition, in \cite[4.7]{Esn96} (in a vague form)  and  in \cite[Intro.]{EV02} (in a more precise form) the problem is posed whether Chern classes of Gau{\ss}-Manin bundles on a smooth variety  $X$ defined over a field of characteristic $0$ vanish in  the Chow groups $CH^n(X)_{\QQ}$ for $n \ge 2$. It is proved to be the case for those of weight $1$;  that is, for the Gau{\ss}-Manin bundles of relative first de Rham cohomology of an abelian scheme over $X$ (\cite[Thm.~1.1]{vdG99}, \cite[Thm.~1.1]{EV02}).  In fact  the weight $1$ Gau{\ss}-Manin bundle is defined on $\mathcal{A}_g$,  which is defined over $\QQ$, and  in \cite[loc.~cit.]{EV02} it is proved that the Chern classes in the Chow groups  of the Deligne extension on the toroidal compactification of $\mathcal{A}_g$ vanish. Thus this example confirms (in a weak sense)  Beilinson's conjecture as well. 

\medskip

The moduli space  $\mathcal{A}_g$  is a  quasi-projective Shimura variety and 
the weight $1$ Gau{\ss}-Manin bundle on it  is  an automorphic bundle associated to the tautological representation of $Sp(2g)$.  A Shimura variety $_KS(G,X)$  (notation explained below) has a canonical model over its reflex field $E(G,X)$, which is a number field    \cite[Section~14]{Mil05}.  It carries a natural family of 
automorphic vector bundles that are defined on this model and themselves have models over explicit finite extensions of $E(G,X)$ \cite[Thm.~4.8]{Har85} (for unexplained notation see Section \ref{review} below).  An automorphic vector bundle $[\mathcal{E}]_K$ that comes from a representation of $G$ is endowed canonically with a flat connection, and  %\mar{'thus' erased, down 2 replaced by 1} 
its Chern classes in Deligne cohomology  $H^{2n}_{\mathcal{D}}(_K S(G,X), \QQ(n))$ vanish for $n\ge 1$ (Theorem~\ref{thm:vanD}, 
Remark~\ref{rmks:rez} 1)). Thus, at least when $_KS(G,X)$ is projective, Beilinson's conjecture implies that the Chern classes vanish even in $CH^n(_KS(G,X))_{\QQ}$ for $n\ge 1$. Unfortunately, we can not prove this.  Instead we prove
\begin{thm} \label{thm:mainthm}
If $_KS(G,X)$ is projective, the Chern classes of an automorphic bundle attached to a representation of $G$ vanish in continuous $\ell$-adic cohomology $H^{2n}_{{\rm cont}}(_KS(G,X), \QQ_\ell(n))$ for $n\ge 1$.

\end{thm}
Stated differently, we prove Conjecture~\ref{conj:sasha} in this particular case. 
The proof relies strongly on the purely algebraic definition of the automorphic bundles, as being associated to a representation of $G$.  Indeed, all automorphic bundles, seen in the category of vector bundles on the Shimura variety,  are eigenvectors for the so-called volume character of the Hecke algebra. The 
Hecke algebra acts semi-simply 
on (continuous) $\ell$-adic cohomology, and the corresponding eigenspace $H^j(_K S(G,X)_{\bar \QQ}, \QQ_\ell)_v$  in $\ell$-adic cohomology 
 identifies with $\ell$-adic  cohomology  $H^j(\hat{X}_{\bar \QQ}, \QQ_\ell)$ of the compact dual $\hat{X}$  of $X$, which itself is generated by algebraic cycles.  This allows us to compute the invariant in  $i$-th Galois cohomology of $ 
 H^{2n-i}(_KS(X,G)_{\bar \QQ}, \QQ_\ell(n))$ (Theorem~\ref{thm:vanl}). 
 
 \medskip
 
 Finally we remark that  if every $\QQ$-simple factor of $G$ has real rank at least $2$, then the super-rigidity theorem of Margulis, applied to the connected components of  $_K S(G,X)$, implies that every flat vector bundle over $_K S(G,X)$ becomes isomorphic to an automorphic vector bundle after replacing $K$ by an appropriate subgroup of finite index.  Since the non-vanishing of Chern classes in the cohomology theories considered here is stable under finite coverings, this implies that for most Shimura varieties the vanishing holds for all flat vector bundles.

\medskip

{\it Acknowledgements:} It is a pleasure to thank Alexander Beilinson for email exchanges on his conjecture~\cite[Conj.~2.4.2.1]{Bei85}, most particularly on the much weaker version Conjecture~\ref{conj:sasha}. This has been a guiding line for formulating the problem on automorphic bundles.  We thank Spencer Bloch for discussions on Conjecture~\ref{conj:sasha}. We thank Ben Moonen for enlightening  exchanges on Conjecture~\ref{mainconj}, which is  a particular instance of Beilinson's conjecture, and Bruno Klingler for Lemma~\ref{kling}.

\subsection*{Conventions}
If $G$ is a reductive algebraic group over $\QQ$, by an {\it admissible irreducible representation} of $G(\ad)$ we will mean an irreducible admissible $(\fg,K)\times G(\af)$-module, where $\fg$ is the complexified Lie algebra of $G$, $K \subset G(\RR)$ is a connected subgroup generated by the center of $G(\RR)$ and a maximal compact connected subgroup,  $G(\af )$ is the group of finite ad\`eles of $G$.  If $\pi$ is such a representation then we will write
 $$\pi \simeq \pi_{\infty} \otimes \pi_f$$
 where $\pi_{\infty}$ is an irreducible admissible $(\fg,K)$-module and $\pi_f$ is an irreducible admissible representation of $G(\af)$.

\section{Automorphic vector bundles and flag varieties}

\subsection{Review of automorphic vector bundles}\label{review}

Let $(G,X)$ be a {\it Shimura datum}, in other words a datum defining a Shimura variety.  We recall that this means that $G$ is a connected reductive group over $\QQ$ and that $X$ is a $G(\RR)$-conjugacy class of homomorphisms $h: \mathbb{S} \ra G_{\RR}$ of real groups, where $\mathbb{S} = R_{\CC/\RR}\mathbb{G}_{m,\CC}$ is $\CC^{\times}$ viewed as an algebraic group over $\RR$.  The pair $(G,X)$ must satisfy a list of familiar axioms that guarantee that $X$ is a $G(\RR)$-equivariant finite union of hermitian symmetric spaces for the identity component of the derived subgroup of $G(\RR)$; see \cite{Mil05}.   In particular, we include the axiom that guarantees that the maximal $\RR$-split torus in the center of $G$ is also split over $\QQ$; without this hypothesis the construction of automorphic vector bundles, as in \eqref{autovb}, is not strictly true as stated, although there are ways to fix this.  Then for any open compact subgroup $K \subset G(\af)$, the double coset space
$$_KS(G,X) = G(\QQ)\backslash X \times G(\af)/K$$
is canonically the set of complex points of a quasiprojective algebraic variety that has a canonical model over a number field, usually denoted $E(G,X)$ and called the {\it reflex field} of $(G,K)$.   If $K' \subset K$ is a subgroup of finite index then the natural map
$$\pi_{K,K'}:  _{K'}S(G,X)  \ra _KS(G,X)$$
is finite; if $K'$ is a normal subgroup of $K$ then $\pi_{K,K'}$ is the quotient map for the  action of the group $K/K'$ on the right.   Moreover, if $K$ is sufficiently small ({\it neat}, in the sense of \cite{Pin90}), then the map $\pi_{K,K'}$ is  finite \'etale.

The precise nature of the canonical model will not be considered in this paper; we will be concerned with $_KS(G,X)$ as a complex algebraic variety, and our aim  is to study the Chern classes of a class of vector bundles on $_KS(G,X)$ that are defined canonically by reference to the origin of the variety in linear algebra.  To this end, choose a base point $h \in X$ and let $K_h \subset G(\RR)$ be its stabilizer.   Then $K_h$ is the group of real points of a reductive subgroup of $G$, which we also denote $K_h$.  We make the useful assumption that $K_h$ is defined over a number field $E_h$; this is always possible, and we may even assume that $E_h$ is a CM field and that every irreducible representation of $K_h$ is defined over a CM field.  In any case, $K_h$ is reductive and there is a natural maximal parabolic subgroup $P_h \subset G_{\CC}$ that contains $K_{h,\CC}$ as a Levi factor.  Let $\hat{X} = G/P_h$ be the corresponding flag variety.  We view $X$ as an analytic open subset of $\hat{X}$ by means of the Borel embedding $\beta$ (this determines the choice of $P_h$ among the two maximal parabolics containing $K_h$). In particular, the complex dimension of this analytic variety is the same as that of $\hat{X}$, which is the same as that of $_KS(G,X)$. 
 Moreover,  $h$ may be viewed as a point of $\hX$.    
 
 Let $E'_h \supset E_h$ denote a finite extension over which $K_h$ becomes a split reductive group.  Then every representation of $K_h$ has a model over $E'_h$.

For any variety $Z$ over $\CC$,  let $\Vect(Z)$ denote the exact category of complex vector bundles on $Z$.    Let $\Vect_G(\hat{X})$ be the category of  $G$-equivariant vector bundles on $\hX$ with coefficients in $\CC$; let 
$$f: \Vect_G(\hat{X}) \ra \Vect(\hX)$$ be the forgetful functor.  Let $\Vect^{ss}_G(\hat{X}) \subset \Vect_G^{\hat{X}}$ denote the subcategory of semisimple $G$-equivariant bundles.  If $H$ is an algebraic group over a ring $R$ and $k \supset R$ is another ring, let $\Rep_k(H)$ denote the category of algebraic representations (of finite type) of $H$ on free modules over $k$.  
There is an equivalence of symmetric monoidal categories 
\ga{repP}{ r_P:  \Vect_G(\hat{X}) \simeq \Rep_{\bar \QQ}(P_h)}
given by taking a vector bundle $B/\hat{X}$ to its fiber $B_h$ at $h$, with the isotropy representation of the stabilizer $P_h$.    Similarly, let 
 $\Vect^{ss}_G(\hat{X}) \subset \Vect_G^{\hat{X}}$ denote the subcategory of semisimple $G$-equivariant vector bundles on $\hat{X}$.  Then \eqref{repP} 
 restricts to an equivalence of symmetric monoidal categories 
\ga{repK}{ r:  \Vect^{ss}_G(\hat{X}) \simeq \Rep_{\bar \QQ}(K_h)}
Evidently we have  canonical isomorphisms
\begin{equation}\label{grading}  K_0(\Vect^{ss}_G(\hat{X})) \isoarrow K_0(\Vect_G(\hat{X})); \ \ K_0(\Rep_{\bar \QQ}(K_h)) \isoarrow K_0(\Rep_{\bar \QQ}(P_h))\end{equation}
compatible with the isomorphisms \eqref{repP} and \eqref{repK}.

\begin{lem}\label{rational}  Every simple object in $\Vect_G(\hat{X})$ has a model over $E'_h$.  
\end{lem}
\begin{proof}  The  fiber  functor \eqref{repK}  at $h$ is evidently rational over the number field $E_h$, so the claim comes down to the assertion that every irreducible representation of $K_h$ has a model over $E'_h$, which we have already noted.
\end{proof} 

On the other hand, for any $K \subset G(\af)$ as above, there is a functor
\begin{equation}\label{autovb} \CE \mapsto [\CE]:  \Vect_G(\hX) \ra \Vect(_KS(G,X))\end{equation} 
defined algebraically in \cite[Thm.~4.8]{Har85}.  As a functor on complex vector bundles we have the explicit construction:
$$[\CE] = [\CE]_K = G(\QQ)\backslash \CE \times G(\af)/K.$$
This is a monoidal functor and it satisfies the following property with respect to change of group:  if $K' \subset K$ then there are canonical isomorphisms
\begin{equation}\label{changegroup} \pi^*_{K,K'}([\CE]_K) \isoarrow [\CE]_{K'}; ~~ \pi_{K,K',*}([\CE]_{K'}) \isoarrow [\CE]_{K}\otimes I_{K'}^K 1,\end{equation}
where $I_{K'}^K 1$ is the  representation of $K$ induced from the trivial representation of $K'$. On the left, this is by definition, and on the right, this is the projection formula.
 In particular, 
\begin{equation}\label{push}\pi_{K,K',*}([\CE]_{K'}) \isoarrow ([\CE]_K)^{[K:K']}\end{equation}
as vector bundles.

%\begin{lem}  

In this paper we work systematically with Chow groups $CH$ and the Grothendieck group $K_0$ of locally free sheaves  with {\it rational} coefficients.  For $H$ and $k$ as above, we let $K_0(\Rep_k(H))$ denote the Grothendieck group of $\Rep_k(H)$, tensored with $\QQ$.
$$ch_{\hX}:  \Vect(\hX) \ra CH(\hX)_{\QQ}, ~~ ch_{K}:  \Vect(_KS(G,X)) \ra CH(_KS(G,X))_{\QQ}$$
denote the respective Chern characters.
 We shall use the following proposition: %\cite[p. 22]{Mar76}%[??]{??}. 

\begin{prop}\label{KhX}  
\begin{itemize}

\item[1)] The map
\ga{}{ ch_{\hX} \circ \circ f \circ r^{-1} :  {\rm Rep}_{\bar \QQ} (K_h)\to \Vect(\hat X)_{\QQ} \to CH(\hat X)_{\QQ} \notag}
 factors through  the composite homomorphism
\ga{}{  K_0(\Rep_\Qbar(K_h))_{\QQ}  \to K_0(\hat X)_{\QQ} \to  CH(\hat X)_{\QQ}. \notag  } 
\item[2)]
The restriction of $ch_{\hX}$ to $\Vect_G(\hat{X})$ generates $CH(\hX)_{\QQ}$.  
\item[3)]
If  we let $\Rep_\Qbar(G) \ra \Rep_{\Qbar}(K_h)$ denote the restriction functor, then  
$ch_{\hX} \circ r^{-1} $ 
induces an isomorphism 
\begin{equation}\label{KCH} K_{\QQ}(\Rep_\Qbar(K_h))\otimes_{K_{\QQ}(\Rep_\Qbar(G))} \QQ \isoarrow CH(\hX)_{\QQ}.\end{equation}
Here the map $K_{\QQ}(\Rep_\Qbar(G)) \ra \QQ$ is given by the augmentation, that is by the rank of a representation.
\end{itemize}
\end{prop}
\begin{proof}  Point 1) is essentially a tautology:  the Chern character obviously factors through $K_0(\hat{X})$ and   $r$ is an exact tensor functor.   Point 2) is the main theorem of \cite{Mar76}.  Suppose $V$ is a representation of $G$; then the corresponding homogeneous bundle on $\hX$ is just $V \times \hX$, with $G$ acting diagonally.  In particular, as a vector bundle it is a sum of $\dim V$ copies of the trivial bundle, hence the restriction to $K_{\QQ}(\Rep_\Qbar(G))$ of the Chern character factors through the augmentation map.  Thus the surjection $ch_{\hX} \circ r^{-1} $ factors through  the left-hand side of \eqref{KCH}.  Now it follows from the main theorem of \cite{Mar76} that this left hand side is of dimension $[W_G:W_{K_h}]$, where $W_G$ (resp. $W_{K_h}$) is the absolute Weyl group of $G$ (resp. $K_h$) relative to a common maximal torus.   On the other hand, the Schubert cells form a basis for the right-hand side, and there are $[W_G:W_{K_h}]$ of them (cf. \cite[3.4.2 (2)]{Bri05}.  So the surjection is an isomorphism by comparing dimensions.  
\end{proof}

  The purpose of the present note is to provide some evidence for the following conjecture, which is an analogue of Proposition \ref{KhX} for Shimura varieties.
  
   \begin{conj}\label{mainconj}  The map $c_K:  \Rep_{\bar \QQ}(K_h) \ra CH(_KS(G,X))_{\QQ}$ defined by
  $$c_K(W) = ch_{K}([r^{-1}(W)])$$
  induces an injective ring homomorphism
  $$K_{\QQ}(\Rep_\Qbar(K_h))\otimes_{K_{\QQ}(\Rep_\Qbar(G))} \QQ \hookrightarrow CH(_KS(G,X))_{\QQ}.$$
  \end{conj}

  We denote by $c_K^{>0}$ the composite of $c_K$ with the projection 
 \ga{}{CH(_KS(G,X))_{\QQ}=\oplus_{n\ge 0} CH^n(_KS(G,X))_{\QQ}\to \oplus_{n>0} CH^n(_KS(G,X))_{\QQ}.  \notag}

%\mar{see the changes down}
\begin{claim}   \label{claim:eq} Conjecture~\ref{mainconj} is equivalent to   
$$c^{>0}_K|_{ \Rep_{\bar \QQ}(G) }=0.$$
\end{claim}
\begin{proof}
 Indeed this condition is equivalent to saying  that $c_K$ induces a homomorphism $$K_{\QQ}(\Rep_\Qbar(K_h))\otimes_{K_{\QQ}(\Rep_\Qbar(G))} \QQ \to CH(_KS(G,X))_{\QQ}.$$ On the other hand, given the Grothendieck-Riemann-Roch theorem, to say that it is injective is equivalent to saying that  the ring homomorphism 
$$K_0(\Rep_{\bar \QQ}(K_h))_{\QQ}\otimes_{K_{\QQ}(\Rep_\Qbar(G))} \QQ  \to K_0(_KS(G,X))_{\QQ}$$  induced by the functor  $K_0(r^{-1})$  is injective. This is true, as follows from Proposition~\ref{isomQ} and point 3) of Proposition \ref{KhX}.
%the isomorphism 
%$\Rep(K_h) \cong K_0(\hat X)$ (see for example \cite{Mar76,Bri05} used in its proof).   % \mar{what does 'its proof'   mean? Do you have a more precise reference?}
\end{proof}

  \begin{rmks} \label{rmks:b}
  
  1) The only instance for which one knows that Conjecture~\ref{mainconj} is true is when $_KS(G,X)$ is the Siegel domain $\CA_g$ and the representation of $G=Sp(2g)$ is the tautological one (see \cite[Thm.~1.1]{EV02}, \cite[Thm.~1.1]{vdG99}). Then the flat vector bundle $[\CE]$ on $\CA_g$ is the Gau{\ss}-Manin bundle of the relative de Rham cohomology  $H^1$ of the universal abelian scheme. The family is defined only over a level structure,  but the Gau{\ss}-Manin bundle, together with its Gau{\ss}-Manin connection, descends to $\CA_g$.  In this case the vanishing is even stronger: on the finite cover over which the local monodromies are unipotent, the Deligne extension of the Gau{\ss}-Manin bundle has vanishing Chern classes in the rational Chow groups.

 \medskip
\noindent  
2)  To tie up with the conjecture on flat bundles alluded to in the introduction, we  remark that the automorphic vector bundles  $[\mathcal{E}]_K$ in the conjecture are those coming from ${\rm Rep}_{\bar \QQ}(G)$ which have a $G(\mathbf{A}_f)$-equivariant integrable connection 
\cite[Lemma~3.6]{Har85}.

 \medskip
\noindent 
  3)   
  Let $c_{\mathcal{D}}: CH^*(_KS(G, X))_{\QQ} \to H^{2*}_{\mathcal{D}}(_KS(G, X),*)$ be  the cycle homomorphism into Deligne cohomology \cite[Section~7]{EV88}.  We prove in Theorem~\ref{thm:vanD}  that
%   \mar{I removed 'projective'}  
   $$c_{\mathcal{D}}\circ c^{>0}_K|_{ \Rep_{\bar \QQ}(G) }=0.$$
  Using Claim~\ref{claim:eq} in addition, one sees that
   Conjecture~\ref{mainconj}  when $_KS(G, X)$ is projective  is a special case of Beilinson's motivic conjecture discussed in the introduction.

  \medskip
  \noindent
  4)  We are not able to prove Conjecture~\ref{mainconj}. In fact, apart from the  example mentioned in 1), we can not prove vanishing in any other example. 
  
  \medskip
  \noindent
  Let us denote by $_KS(G,X)_{E(G,X)}$ the model of the Shimura variety over its reflex field. 
 If $$c_{\ell}:    CH^*(_KS(G, X)_{E(G,X)})_{\QQ} \to H^{2*}_{\rm cont}(_KS(G, X)_{E(G,X)}, \QQ_\ell(*))$$  denotes the cycle homomorphism to continuous $\ell$-adic cohomology \cite[Section~2]{Jan87}, we prove
   $$c_{\ell}\circ c^{>0}_K|_{ \Rep_{\bar \QQ}(G) }=0.$$
    in  Theorem~\ref{thm:vanl}.
(More precisely, we prove this after replacing $E(G,X)$ by the finite extension $E'_h$ of Lemma \ref{rational}.)  In particular, we verify Conjecture~\ref{conj:sasha} in this case. To our knowledge, the examples treated in this note are the first that confirm this prediction for cycle classes of flat  bundles 
that do not depend on knowing in advance that the Chow class  itself vanishes.   
   
  \end{rmks}

 \subsection{Hecke operators} 
 
 Fix $K \subset G(\af)$.  Let $g \in G(\af)$ %\mar{changed $\gamma$ to $g$ here, is this right?} 
 and consider $K_g = K \cap gKg^{-1} \subset K$.  Let 
 $$\pi_{1,g} = \pi_{K,K_g}:  _{K_g}S(G,X)  \ra _KS(G,X),$$
 defined as above.  Right multiplication by $g$ defines an isomorphism 
 $$r_g:  _{gKg^{-1}}S(G,X)  \isoarrow _KS(G,X).$$
 Let
 $$\pi_{2,g} = r_g\circ \pi_{gKg^{-1},K_g}:  _{K_g}S(G,X)  \ra _KS(G,X).$$
 We first observe that 
 \ga{TgE}{ T(g)[\CE]_K \cong\oplus_1^{ [K:K_g]}[\CE]_K}
 as vector bundles, 
 where $T(g)= \pi_{2,g*}\circ \pi^*_{1,g}$. Indeed, the definition implies $r_g^* [\CE]_{K}\cong [\CE]_{g^{-1}Kg}$, and  formula~\eqref{changegroup} implies formula~\eqref{TgE}.

Both $\pi_{1,g}$ an $\pi_{2,g}$ are finite \'etale morphisms. 
 So for any contravariant  cohomology theory $H$ which has  push-downs for proper (or even only finite \'etale) morphisms, one can define 
 the {\it Hecke operator}
 \ga{tg}{   T(g): H (_KS(G,X))  \xrightarrow{ \pi_{2,g*}\circ \pi_{1,g}^*} H(_KS(G,X)).}

 We shall use the Hecke operators on Chow groups, on continuous $\ell$-adic cohomology, on Deligne cohomology, on syntomic cohomology.  All of them are considered rationally.  In particular, they are (possibly infinite) dimensional vector spaces over $\QQ$, and the Hecke algebra splits those cohomologies as a sum of generalized eigenspaces.   
 
 \medskip
 
 A rational prime number $q$ is {\it unramified} for $K$ if there exists a $K_q \subset K$ with $K_q \subset \cap G(\QQ_q)$ a hyperspecial compact open subgroup, and {\it ramified} otherwise.  There is a finite set $S(K)$ of ramified primes.    We let $\CH_K$ denote the $\QQ$-subalgebra of  the ring tensor $\QQ$ of correspondences  generated   by the $T(g)$,  where $g$ runs through elements of $G(\QQ_q)$ with $q \notin S(K)$; this is well-known to be a commutative algebra.   

 The following is obvious:
 
  \begin{lem}  Let $R$ be a ring, $a \in R$, and let $[a]:  _KS(G,X) \ra R$ be the constant function with value $a$.  Then $[a]$ is an eigenfunction for every $T(g)$, with eigenvalue $v(K,g) = [K:K_g]$.
 \end{lem}
 
 For any cohomology theory $H^*(_KS(G,X))$, including the Chow groups, let $H^*(_KS(G,X))_v \subset H^*(_KS(G,X))$ be the  
 eigenspace for the {\it volume character} $T(g) \mapsto v(K,g)$ of $\CH_K$. As $H^*(_KS(G,X))$
 is a possibly infinite dimensional $\QQ$-vector space, so is $H^*(_KS(G,X))_v$.

 \begin{lem}  Let $\CE \in \Vect_G(\hX)$.  Then for any open compact subgroup $K \subset G(\af)$, $ch([\CE]_K) \in CH(_KS(G,X))_v$.
 \end{lem}
 \begin{proof}  
 As $\pi_{2,g}$ is finite \'etale,  its Todd class is equal to $1$, thus the Grothendieck-Riemann-Roch theorem implies 
 \ga{}{ ch(\pi_{2,g*} \pi_{1,g}^*[ \CE]_K)=  \pi_{2,g*}(ch( \pi_{1,g}^*[\CE]_K)\cdot 1), \notag}
 from which we conclude using formula~\eqref{TgE}
\ml{}{ ch(T(g) [\CE]_K)= \pi_{2,g*}(ch( \pi_{1,g}^*[\CE]_K)\cdot 1)=  \pi_{2,g*}\pi_{1,g}^*ch([\CE]_K) =T(g)ch([\CE]_K).\notag}
 Applying   \eqref{changegroup} and \eqref{push} one concludes 
 \ga{tgch}{ [K:K_g] ch([\CE]_K)=   T(g)ch([\CE]_K)   . } %\mar{added everywhere $[-]_K$}
 \end{proof}
 
 \begin{cor}  \label{lem:volcoh}  Let $H^*$ be a cohomology theory which,
for  any open compact subgroup $K \subset G(\af)$, 
 admits 
 a cycle map $$c^K_H: CH(_KS(G,X))\to H(_KS(G,X))$$ which commutes with the action of $\CH_K$. 
 Then for any $\CE \in \Vect_G(\hX)$,  $$c_H^K\circ ch([\CE]_K) \in H^*(_KS(G,X))_v.$$
 \end{cor}

The cohomology theories $H^i(-, j)$  with coefficients in a characteristic $0$ field $F$ considered in this note are  all functorial. Thus to check that the cycle map commutes with the action of $\CH_K$, it suffices to verify  compatibility with the push-down via $\pi_{2,g*}$.
In all those cohomology theories, the cycle map can be defined via purity: for $Z=\sum_i m_i Z_i$ a codimension $n$ cycle on a smooth $Y$, where $m_i\in \ZZ$ and the codimension $n$ cycles $Z_i$ are prime, the Gysin morphism exists and is an isomorphism
\ga{}{\gamma: \oplus  K\cdot [Z_i] \xrightarrow{\cong} H_Z^{2n}(Y, n). \notag}
See \cite[Section~7]{EV88} for Betti, de Rham and Deligne cohomology, 
\cite[Thm.~3.23]{Jan88} for continuous $\ell$-adic cohomology. Thus compatibility reduces to showing 
\ml{}{ \gamma(\pi_*[Z_i])=\gamma({\rm deg}(k(Z_i)/k(\pi(Z_i)))\cdot [\pi(Z_i)])= \\ \pi_* \gamma([Z_i])
\in H^{2n}_{p(Z_i)} (Y', n) \notag} for a finite  surjective morphism $p: Y\to Y'$.  In Deligne cohomology, it follows as the cohomology verifies the Bloch-Ogus axioms. In $\ell$-adic cohomology, for lack of reference, we restrict ourselves to the case where  $Y$ is defined over a number field $k$.
Let $\mathcal{Y}$ be a flat smooth model over a non-trivial open $U$ in the spectrum of the ring of integers of $k$. Then the cycle class in $H^{2n}(\mathcal{Y}, n)$ is just the standard cycle class from \cite{SGA4.5}, for which the trace properties are known  \cite[2.3]{SGA4.5}.
%\cite[Thm.~2.1.1]{Fuj02} for \'etale cohomology. 

\begin{cor} \label{lem:volcoh2}
Corollary~\ref{lem:volcoh} holds true for Deligne cohomology and continuous $\ell$-adic cohomology.

\end{cor}
 In the next section we study $H^*(_KS(G,X))_v$ for Betti and $\ell$-adic cohomology. 
 
  \subsection{Chern classes in cohomology}
  
  Henceforward we assume  the Shimura variety $_KS(G,X)$ to be  {\it projective}; equivalently, the derived subgroup $G^{{\rm der}}$ of $G$ is anisotropic over $\QQ$.  In the non-compact case the automorphic theory naturally gives information about Chern classes of canonical extensions on toroidal compactifications on the one hand; on the other hand, the $v$-eigenspace most naturally appears in intersection cohomology of the minimal compactification.  This has been worked out in detail by Goresky and Pardon in \cite{GP02}, and we expect 
  to study the
 analogous questions for Chow groups in a second paper.  %\mar{{\color{red} do we keep the mention of a (yet?) non-existing paper?}}
 
 Although the point of the Goresky-Pardon paper is to study non-compact Shimura varieties, it still contains a convenient reference for our purposes.  The following statement is well known.
 
 \begin{prop}\label{isomQ}   Assume the derived subgroup $G^{{\rm der}}$ of $G$ is anisotropic, so that $_KS(G,X)$ is projective.   There is a canonical isomorphism of algebras
 $$H^*(\hX,\QQ) \isoarrow H^*(_KS(G,X),\QQ)_v.$$
 
 \end{prop}
 \begin{proof}  We first show the corresponding statement over $\CC$; thus we can compute $H^*(_KS(G,X),\CC)$ using automorphic forms and Matsushima's  formula.  Say
 the space $\CA(G)$ of automorphic forms on $G(\QQ)\backslash G(\ad)$ decomposes as the direct sum
 $$\CA(G) = \oplus_{\pi} m(\pi)\pi$$
 where $\pi$ runs over irreducible admissible representations of $G(\ad)$ and $m(\pi)$ is a non-negative integer, which is positive for a countable set of $\pi$.  Then
 $$H^i(_KS(G,X),\CC) \isoarrow \oplus_{\pi} m(\pi)H^i(\fg,K_h; \pi_{\infty})\otimes \pi_f^{K}.$$
 Then 
 $$H^i(_KS(G,X),\CC)_v \isoarrow \oplus_{\pi} m(\pi)H^i(\fg,K_h; \pi_{\infty})\otimes (\pi_f^{K})_v,$$
 where $(\pi_f^{K})_v$ is the eigenspace in $\pi_f^K$ for the volume character  of $\CH_K$.  Write $\pi_f = \otimes'_q \pi_q$, where $q$ runs over rational primes.  Now if $q$ is unramified for $K$ then $\pi_f^K = 0$ unless $\pi_q$ is spherical; but the only spherical representation of $G(\QQ_q)$ whose spherical subspace is an eigenspace for the (local) volume character is the  {\it trivial} representation of $G(\QQ_q)$.  Thus   $(\pi_f^{K})_v = 0$ implies $\pi_q$ is trivial representation for all $q$ that are unramified for $K$.  It then follows from weak approximation that $\pi$ is in fact the trivial representation.  Thus for all $i$,
 \begin{equation}\label{HHv} H^i(_KS(G,X),\CC)_v \isoarrow H^i(\fg,K_h; \CC).\end{equation}
 But this is equal to $H^i(\hX,\CC)$ by a standard calculation; see \cite[Rmk.~16.6]{GP02}.
 
 In particular, $H^i(_KS(G,X),\QQ)_v = 0$ if $i$ is odd.  Now to prove that there is an isomorphism over $\QQ$, it suffices to show that, for each $m$,
\begin{equation}\label{twist} H^{2m}(\hX,\QQ)(m) \isoarrow H^{2m}(_KS(G,X),\QQ)(m)_v,\end{equation}
where we write  $$H^{2m}(\hX,\QQ)(m) = H^{2m}(\hX,  (2\pi i)^m\cdot \QQ) \subset H^{2m}(\hX,\CC).$$  But composing the isomorphism of Proposition \ref{KhX} with the Chern class in cohomology, we obtain an isomorphism \cite{Mar76},  \cite{GP02} %\mar{precise ref. in the articles?}
\begin{equation}\label{HHX} K_{\QQ}(\Rep_\Qbar(K_h))\otimes_{K_{\QQ}(\Rep_\Qbar(G))} \QQ \isoarrow H^{2m}(\hX,\QQ)(m). \end{equation}
Then \eqref{twist} follows from the diagram in \cite[Rmk.~16.6]{GP02}, where of course we are replacing intersection cohomology with ordinary cohomology.
 \end{proof}
    
%{\color{red} Cite Goresky-Pardon, Remark 16.6, in the case of compact quotient!  Show that the image {\it equals} $H^*(_KS(G,X))_v$ in the case of compact quotient!}
  
% \section{Arthur's multiplicity conjectures} 
The proof of Proposition \ref{isomQ} gives additional information on the Galois action on $\ell$-adic cohomology. The Shimura variety $_KS(G,X)$ and the flag variety $\hX$ have canonical models over the {\it reflex field} $E(G,X)$.   Theorem~\cite[Thm.~4.8]{Har85} asserts, among other things,  that the functor of \eqref{autovb} commutes with the action of $\Gal(\Qbar/E(G,X))$.  In addition, the Hecke correspondences $T(g)$ are all rational over $E(G,X)$.  It follows that, for any prime $\ell$, the $\ell$-adic cohomology spaces $H^*(\hX_{\bar \QQ},\Ql)$ and $H^*(_KS(G,X)_{\bar \QQ},\Ql)_v$ carry an action of $\Gal(\Qbar/E(G,X))$, where we denote by $_{\bar \QQ}$ the base change of the models of the Shimura variety and of the flag variety  over the reflex field $E(G, X)$. 

Moreover, for any $\CE \in \Vect_G(\hX)$, the $i$-th Chern class
\ga{}{ c^i_\ell(\CE) \in H^{2i}(\hX_{\bar \QQ},\Ql (i)),  \ {\rm resp.} \   c^i_\ell([\CE]_K) \in H^{2i}(_KS(G,X)_{\bar \QQ},\Ql(i))_v \notag}
 generates a $\Gal(\Qbar/E(G,X))$-subspace that is isotypic for the $i$-th power of the cyclotomic character. %\mar{last piece removed: 'in other wrods'...as over $E(G,X)$ the bundles are not rational thus Galois of $E(G,X)$ rather acts as a permutation}
 %;  in other words, it is a copy  of the Tate representation $\Ql(0)$.  

\begin{prop}  \label{prop:corr} Under the hypotheses of Proposition \ref{isomQ}, the algebra isomorphism induces an algebra  isomorphism 
\ga{}{H^{2*}(\hX_{\bar \QQ},\Ql(*)) \isoarrow H^{2*}(_KS(G,X)_{\bar \QQ} ,\Ql(*))_v \notag \\
c^i_\ell(\CE) \mapsto c^i_\ell([\CE]_K) \notag
}
in $\ell$-adic cohomology over $\bar \QQ$, which  is equivariant for the action of the Galois group $\Gal(\Qbar/E(G,X))$ on both sides.
\end{prop} 
\begin{proof}  Both sides are generated by Chern classes $c^i_\ell(\CE) $ and $c^i_\ell([\CE]_K)$, and the isomorphism is defined on those Chern classes.
\end{proof}

\section{Chern classes in Deligne cohomology}\label{Deligne}
We use the notations from Remarks~\ref{rmks:b} 2).  First we recall some of the basic properties of automorphic vector bundles in the image of $\Rep_\Qbar(G)$.

\begin{prop} \label{prop:hodge} Let $i:  \Rep_\Qbar(G) \ra \Vect_G(\hX)$ be the composition of the natural inclusion $\Rep_\Qbar(G) \ra \Rep_\Qbar(K_h)$ with the inverse of the equivalence \eqref{repK}.   Let $\rho:  G \ra GL(V)$ be a finite-dimensional representation.  Then
\begin{itemize}
\item[(a)]  The vector bundle $[i(\rho)]$ on $_KS(G,X)$ has a canonical flat connection.
\item[(b)]  Let $[i(\rho)]^{\nabla}$ denote the local system on $_KS(G,X)$ that corresponds to the flat connection of (a).  Let $Z \subset _KS(G,X)$ be a connected component, let $z \in Z$ be a base point, and let $r_z:  \pi_1(Z,z) \ra V$ be the monodromy representation attached to $[i(\rho)]^{\nabla}$.  Then $V$ has a model over the integers $\CO_F$ of a number field $F$ that is preserved by the image of $r_z$.
\item[(c)]  If $\rho$ is defined over $\QQ$ and  then $[i(\rho)]$ is endowed with a canonical variation of Hodge structure, which is a direct sum of variations of pure Hodge structures.
\end{itemize}
\end{prop}   
\begin{proof}  These points are all well-known.  For (a), one can cite \cite[Lem.~3.6]{Har85}; for (c) see  of \cite[Section~1.1]{Del79}, especially 1.1.13-1.1.17.    For (b) we use Lemma~\ref{kling}, the  proof  of which was provided by Bruno Klingler.   Indeed, the complex variety $Z$ in (b) is the quotient of a connected component $X^0$ of $X$, which is contractible, by a congruence subgroup $\Gamma \subset G(\QQ)$, and the topological fundamental group  $\pi_1(Z,z)$ can be identified with $\Gamma$.   It suffices to assume $(r_z,V)$ is irreducible.  The reductive group $G$ splits over a number field $F$, and by the theory of Chevalley groups, every irreducible representation of $G$ has a model over $F$.  Then $R_{F/\QQ}(r_z,V)$ is a representation defined over $\QQ$, to which Lemma~\ref{kling} applies.   The lattice $L$ of Lemma~\ref{kling} generates an $\CO_F$ lattice in $V$ which is invariant under $\Gamma$.  
\end{proof}

\begin{lem}\label{kling}  Let $G$ be a linear algebraic group over $\QQ$ and $\Gamma \subset G(\QQ)$ an arithmetic subgroup.  Let $W$ be a vector space over $\QQ$ and let $r:  G \ra GL(W)$ be a representation defined over $\QQ$.  Then there exists a lattice $L$ in $W$ such that $r(\Gamma) \subset GL(L)$.  \end{lem}

\begin{proof}  Let $H = r_z(G) \subset GL(W)$.  By  \cite[Thm.~1.4]{PR94}, $r_z(\Gamma)$ is an arithmetic subgroup of $H$.   The Lemma then follows from  \cite[Prop.~4.2]{PR94}, which is the special case where $r$ is injective.
\end{proof}

\begin{thm}\label{thm:vanD}
Let $_KS(G,X)$ be a  Shimura variety. Then 
\ga{}{  c_{\mathcal{D}} \circ c^{>0}_K|_{ \Rep_\CC(G) }=0. \notag}

\end{thm} 
\begin{proof} Given Proposition~\ref{prop:hodge}, the theorem
 is a direct consequence of \cite[Thm.~0.2]{CE05}, which in fact says more: the Chern-Simons invariants of the flat connection on $[\CE]$ are torsion, for $\CE$ coming from $\Rep_{\mathbb{C}}(G)$.

\end{proof}
However, for projective Shimura varieties,  the theorem is a consequence of Lemma~\ref{lem:volcoh2}
and of the following proposition.

We denote by H$_{dR}$ de Rham cohomology. 
\begin{prop} \label{prop:oddD}
Let $_KS(G, X)$ be a  projective Shimura variety. Then 
$$H^{2n}_{\mathcal{D}}( _KS(G, X), \QQ(n))_v  \subset H^{2n}_{dR}( _KS(G, X)).$$

\end{prop}
\begin{proof}
Proposition~\ref{isomQ} implies that $H^{2n-1}(_KS(G, X), \QQ)_v=0$. On the other hand, the action of $\mathcal{H}_K$ on $H^{2n}_{\mathcal{D}}( _KS(G, X), \QQ(n))$ is via correspondences, is semi-simple. Thus it respects the exact sequence 
\ga{}{ 0\to H^{2n-1}(_KS(G, X), \CC)/\big( H^{2n-1}(_KS(G, X), \QQ(n))+ F^n\big) \notag \\ \to H^{2n}_{\mathcal{D}}( _KS(G, X), \QQ(n)) \to H^{2n}_{dR}( _KS(G, X)) \notag}
and on the left each of the three terms of the quotient. This finishes the proof.
\end{proof}

\begin{rmks} \label{rmks:rez} 
1) We have assumed $_KS(G, X)$ projective in order to apply Proposition ~\ref{isomQ} as stated above.  For general Shimura varieties the map in Proposition ~\ref{isomQ} is in any case surjective -- cf. \cite{GP02} -- and this suffices for the above proposition.  %However, there is more than one way to define Chern classes for non-compact Shimura varieties and it would be desirable to understand how far vanishing holds in the non-compact case.  
%\mar{I removed 'However'.. since we have already the mention p.11 1.3 end of first section.}

\medskip
\noindent
2) Theorem~\cite[Thm.~0.2]{CE05} used in the proof of Theorem~\ref{thm:vanD} is a variant of Reznikov main theorem \cite{Rez94} which also rests on the fact that some oddly weighted forms do not exist.

\end{rmks}

\section{Chern classes in continuous $\ell$-adic cohomology}\label{elladic} % \mar{in the whole section here one should replace the reflex field by the finite extenion on which the autom blds are defined, see your intro. Please choose a notation for this and we change the reflex field all along down by this field}
Recall the field $E'_h$ introduced in \S 1.  It follows from Lemma \ref{rational} that every $\ell$-adic Chern class $c_K(W)$  belongs to 
$CH^*(_KS(G, X)_{E'_h})$.  Thus we can define the class $c_{\ell }\circ c_K(W) \in H^{2*}_{\rm cont}(_KS(G, X)_{E'_h}, \QQ_\ell(*))$.

\begin{rmk}  The discussion of the present section seems to depend on the choice of a CM point $h$.  In fact, it is not difficult to see that, if we let $E(G)$ denote the extension of $\QQ$ over which $G$ splits, we can replace $E'_h$ by $E(G)E(G,X)$, and Lemma \ref{rational} remains true.  The point is that, up to twisting by a power of the canonical bundle of $\hat{X}$, every irreducible $G$-equivariant vector bundle on $\hat{X}$ can be obtained, by a construction that is defined over $E(G,X)$, as a canonical quotient of  a $G$-equivariant vector bundle attached to an irreducible representation of $G$.  This is a simple application of the theory of the highest weight, applied to $K_h$ for variable $h \in \hat{X}$, and to $G$.  Since for our purposes it suffices to prove the vanishing in rational continuous $\ell$-adic cohomology after restriction to some finite extension, the choice of number field is immaterial.
\end{rmk}
  
We use the notation of Remarks~\ref{rmks:b} 3), 4). The action of the Hecke algebra $\mathcal{H}_K$ commutes with the Galois action of $\Gal(\bar \QQ/E(G,X))$. 
The Hecke algebra splits \'etale cohomology $H^i(_KS(G,X)_{\QQ}, \QQ_\ell)$ into a sum of generalized eigenspaces which are Galois invariant. Thus $\mathcal{H}_K$ splits the filtration stemming from the Hochshild-Serre spectral sequence.
In particular one has a filtration on $ H^{2n}_{\rm{cont}}( _KS(G, X)_{E'_h}, \QQ_\ell(n)))_v  $ with $0$-th graded quotient equal to $H^0(E'_h,    H^{2n}( _KS(G, X)_{\bar \QQ}, \QQ_\ell(n))_v) $, first graded 
quotient equal to $H^1(E'_h,    H^{2n-1}( _KS(G, X)_{\bar \QQ}, \QQ_\ell(n))_v) $ and last 
graded 
quotient  being the subspace  
\ga{}{ H^2(E'_h,    H^{2n-2}( _KS(G, X)_{\bar \QQ}, \QQ_\ell(n))_v)  \subset    H^{2n}_{\rm{cont}}( _KS(G, X)_{E'_h}, \QQ_\ell(n)))_v .\notag}

The pendant on the $\ell$-adic side of Proposition~\ref{prop:oddD}  is
\begin{prop} \label{prop:oddl}
Let $_KS(G, X)$ be a  projective Shimura variety. Then 
\ga{}{  \big(H^{2n}_{\rm{cont}}( _KS(G, X)_{E'_h}, \QQ_\ell(n)))/ 
H^2(E'_h, H^{2n-2}( _KS(G, X)_{\bar \QQ}, \QQ_\ell(n)))\big)_v \notag \\
  \subset H^{2n}( _KS(G, X)_{\bar \QQ}, \QQ_\ell(n)). \notag}

\end{prop}
\begin{proof}
Proposition~\ref{prop:corr} implies that $H^{2n-1}( _KS(G, X)_{\bar \QQ}, \QQ_\ell(n))_v=0$. 

\end{proof}

\begin{thm}\label{thm:vanl}
Let $_KS(G,X)$ be a projective Shimura variety. Then 
\ga{}{  c_{\ell } \circ c^{>0}_K|_{ \Rep_\QQ(G) }=0. \notag}

\end{thm} 
\begin{proof}
By Proposition~\ref{prop:oddl}, together with Corollary~\ref{lem:volcoh2}, 
$c_{\ell } \circ c^{>0}_K|_{ \Rep_\QQ(G)} $ has values in 
$$H^2(E'_h, H^{2n-2}( _KS(G, X)_{\bar \QQ}, \QQ_\ell(n))_v).$$
We now apply a variant of the proof of \cite[Prop.~2.3]{Ras95}. Let us denote by $d$ the dimension of $_KS(G,X)$. Then 
Proposition~\ref{prop:corr} implies  that the non-degenerate  $E'_h$- equivariant cup-product 
\ml{}{ H^{2n-2}(_KS(G,X)_{\bar \QQ},\QQ_\ell( n)) \times 
H^{2d-2n+2}(_KS(G,X)_{\bar \QQ},\QQ_\ell( d-n+1)) \notag \\
\to H^{2d}(_K S(G,X)_{\bar \QQ}, \QQ_\ell(d+1))= \QQ_\ell(1) \notag }
restricts to a non-degenerate  $E'_h$-equivariant cup-product 
\ml{3.3}{ H^{2n-2}(_KS(G,X)_{\bar \QQ},\QQ_\ell( n))_v \times 
H^{2d-2n+2}(_KS(G,X)_{\bar \QQ},\QQ_\ell( d-n+1))_v  \\
\to H^{2d}(_K S(G,X)_{\bar \QQ}, \QQ_\ell(d+1))_v= H^{2d}(_K S(G,X)_{\bar \QQ}, \QQ_\ell(d+1))_v=
 \QQ_\ell(1).  }
 Write 
\ml{}{h={\rm dim}_{\QQ_\ell} H^{2n-2}(_KS(G,X)_{\bar \QQ},\QQ_\ell)_v= \\ {\rm dim}_{\QQ_\ell} 
H^{2d-2n+2}(_KS(G,X)_{\bar \QQ},\QQ_\ell)_v. \notag}
Then  \eqref{3.3}  is written as 
\ga{}{ \oplus_1^h \QQ(1) \times \oplus_1^h \QQ_\ell(0) \to \QQ_\ell(1) \notag}
as  a non-degenerate $E'_h$-equivariant pairing. Indeed,  by Proposition~\ref{prop:corr}, for all $i\in \mathbb{N}$, 
 $H^{2i}(_KS(G,X)_{\bar \QQ},\QQ_\ell( i))_v$ is spanned as a $\QQ_\ell$-vector space by the classes 
$$c^{i}_\ell([\mathcal{E}]_{K, \bar \QQ}) \in H^0(E'_h, H^{2i}(_KS(G,X)_{\bar \QQ},\QQ_\ell( i))),$$  
where $\mathcal{E}$ comes from a $K_h$-representation, so is algebraic. 
This implies that the pairing of $\QQ_\ell$-vector spaces
\ga{3.4}{H^2(E'_h, \oplus_1^h \QQ_\ell(1) )\times H^0(E'_h, \oplus_1^h \QQ_\ell(0)) \to H^2(E'_h, \QQ_\ell(1))}
is non-degenerate. 

\medskip 

On the other hand, for  $W \in {\rm Rep}_{\bar \QQ} (G)$, 
and $E \in {\rm Rep} (K_h)$ defining $[\mathcal{E}]_K$,
 one has 
\ga{}{ c_K^n (W)\cup ch^{d-n+1} ([\mathcal{E}]_K) \in CH^{d+1} (_KS(G,X))_{\QQ}=0. \notag }
(Recall here $c_K$ is defined in Conjecture~\ref{mainconj}.) Thus  
\ml{}{0= c_\ell ^{d+1}(  c_K^n (W)\cup ch^{d-n+1} ([\mathcal{E}]_K)  )= \notag  \\ c^n_\ell(  c_K^n (W)) \cup c^{d-n+1}_\ell([\mathcal{E}]_{K, \bar \QQ} ) \in  \notag%\\  
 H^2(E'_h, \QQ_\ell(1)) 
. \notag}
Applying \eqref{3.4} we conclude $c_K^n(W)=0$. This finishes the proof.
\end{proof}

%\mar{I removed the Rmk  labelled 3.5 earlier on, we had this already}

\begin{rmk}[Syntomic cohomology]
In addition to Deligne and continuous $\ell$-adic cohomology, it is natural to consider syntomic cohomology as defined by Fontaine, Kato and Messing. Let us denote by $E'_{\frak{p}}$ the $p$-adic completion of the number field $E'_h$ at a place $\frak{p}$.  We assume that $_K S(G, X)_{E'_{\frak{p}}}$ is proper and has a semi-stable model. 
 Then with $p$-power torsion coefficients, the syntomic cohomology group 
 $H^{2n}_{\rm synt}(_K S(G, X)_{E'_{\frak{p}}},\ZZ/p^n\ZZ (n))$ is defined  as the \'etale cohomology group of the $\tau_{\le n}$-truncation of the vanishing cycle complex  \cite[Thm.~2.2]{KM92}. This isomorphism lifts to continuous $p$-adic coefficients  \cite[Proof~of~Cor.~4.5]{NN16} yielding a Hochshild-Serre spectral sequence 
\ml{}{E_2^{st}=H^s_{st}(E'_{\frak{p}}, H^t( _K S(G, X)_{\bar \QQ_{\frak{p}}}, \QQ_p(n)) \Longrightarrow\\
H^{s+t}_{\rm synt} (_K S(G, X)_{E'_{\frak{p}}}, \QQ(n)) \notag}
where $st$ stands for `stable', 
compatible with the one on continuous $p$-adic cohomology \cite[Thm.~4.8]{NN16}.  Thus the same proof as in Theorem~\ref{thm:vanl} yields the same result, with $c_\ell$ replaced by the syntomic Chern classes (\cite[Section~5]{NN16}).
\end{rmk}

\begin{rmk}[Construction of torsion classes in cohomology]  One motivation for studying the Chern classes of automorphic vector bundles is the hope that they might provide a way to construct interesting torsion classes in the Chow group, or that their $\ell$-adic Abel-Jacobi classes in $H^1(E'_h,    H^{2n-1}( _KS(G, X)_{\bar \QQ}, \QQ_\ell(n))_v) $ might be torsion.  It follows easily from \eqref{HHv} and \eqref{HHX} that any class in $K_0(\Vect_G(\hat{X}) )\simeq K_0(\Rep_{\QQ}(K_h))$ whose image   in the Chow group $CH(_K S(G, X))$ under all Chern classes %\mar{I changed a bit: the Chern char. does not have values in the integeal $CH$, so has to take Chern classes, and thisis n olonger a character...}  
of positive degree
 is torsion,  must necessarily belong to the ideal generated by the kernel of the augmentation map $K_0(\Rep_{\QQ}(G)) \ra \ZZ$.  Any torsion classes arising in this way would naturally be eigenvectors for the volume character of the Hecke algebra; in particular, the associated cohomology classes are {\it Eisenstein} classes, in the usual sense.
\end{rmk}

%\mar{I removed the section,we have already mentioned the van.of GM of wt 1 in the intro and in Rmks 1.9 1), it is more than enough}

%\keywords{Galois representation, Shimura variety, special values of $L$-functions}
%\end{keywords}

\end{document}